\documentclass[10pt,a4paper]{amsart}
\usepackage{graphicx}
\usepackage{epsfig}
\usepackage{amsmath, amssymb}
\usepackage{placeins}
\usepackage{amsfonts,amsmath, amsthm}
\usepackage{verbatim}
\usepackage{xfrac}
\usepackage{anyfontsize}

\newtheorem{theorem}{Theorem}

\newtheorem{rem}{Remark}[theorem]
\newtheorem{prop}{Proposition}[theorem]

\DeclareMathOperator{\li}{li}
\DeclareMathOperator{\supp}{supp}

\author{Dimitris Vartziotis}
\address{NIKI Ltd. Digital Engineering, Research Center, 205 Ethnikis Antistasis Street,
45500 Katsika, Ioannina, Greece \and \\TWT GmbH Science \& Innovation, Department for Mathematical Research \&
Services, Ernsthaldenstr. 17, 70565 Stuttgart, Germany }
\email{dimitris.vartziotis@nikitec.gr }

\author{Doris Bohnet}
\address{TWT GmbH Science \& Innovation, Department for Mathematical Research \&
Services, Ernsthaldenstr. 17, 70565 Stuttgart, Germany}
\email{doris.bohnet@twt-gmbh.de}

\title[Fractal curves from primes]{Fractal curves from prime trigonometric series}
\begin{document}

\begin{abstract}We study the convergence of the parameter family of series $$V_{\alpha,\beta}(t)=\sum_{p}p^{-\alpha}\exp(2\pi i p^{\beta}t),\quad \alpha,\beta \in \mathbb{R}_{>0},\; t \in [0,1)$$ defined over prime numbers $p$, and subsequently, their differentiability properties. The visible fractal nature of the graphs as a function of $\alpha,\beta$ is analyzed in terms of H\"{o}lder continuity, self similarity and fractal dimension, backed with numerical results. We also discuss the link of this series to random walks and consequently, explore numerically its random properties. \\

\end{abstract}
\maketitle
\section{Introduction}  
The prime numbers are not randomly distributed but, there are random models that capture well important properties of the distribution of prime numbers (e.g. \cite{cramer_1936}). The random behavior of a deterministic mathematical object can be found elsewhere: there are classical function series that can be approximated by random processes. Let us briefly describe these series:\\
consider the two functions $f_n(x)=\sin(2\pi nx)$ and $f_{n+1}(x)=\sin(2\pi (n+1)x)$ for an arbitrary integer $n\in \mathbb{N}$. These behave as strongly dependent random variables if we consider $x$ to be a random real variable uniformly distributed on some interval. But if one picks from the sequence of frequencies $(2\pi nx)_{n\geq 0}$ a sub sequence $(2\pi n_k)_{k\geq 0}$ such that the integer sequence grows sufficiently fast, i.e. $\sfrac{n_{k+1}}{n_k} \geq 1 + \rho, \rho > 0$, the quantities $f_{n_k}(x)$ and $f_{n_{k+1}}(x)$ behave like independent random variables (see Fig.~\ref{fig:introduction} as an example and \S~\ref{sec:random}). 
\begin{figure}
\begin{minipage}{0.46\textwidth}
 \includegraphics[width=\textwidth]{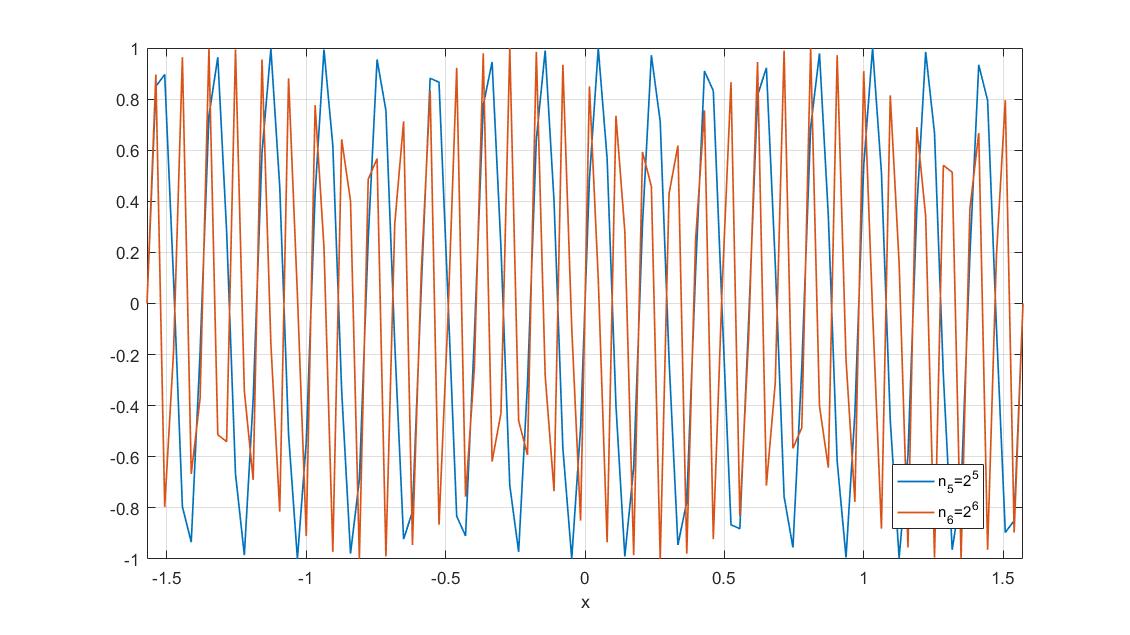}

\end{minipage}  
\begin{minipage}{0.46\textwidth}
 \includegraphics[width=\textwidth]{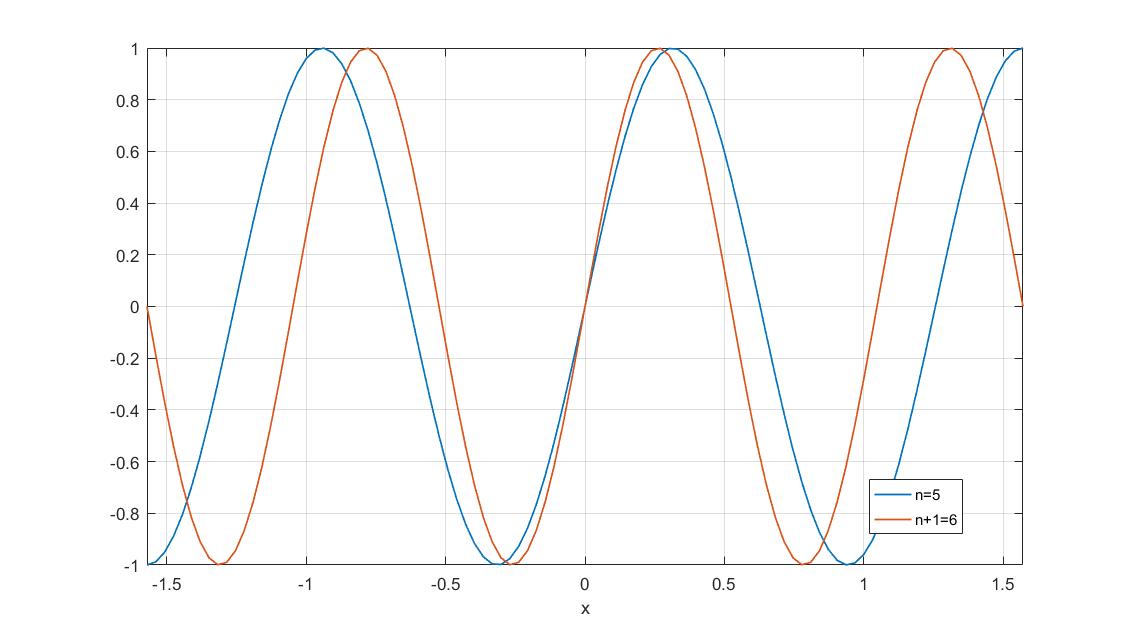}

\end{minipage} 
\caption{Graph of $\sin(2^5\pi x)$ and $\sin(2^6\pi x)$ on the left and of $\sin(5\pi x)$ and $\sin(6\pi x)$ on the right. } 
\label{fig:introduction}
\end{figure}
Now, one can construct a random walk out of these random variables: start at $0$. At time $k$ move $f_{n_k}(x)$ up. At time $N$, we find ourselves at $S(x,N)=\sum_{k=0}^N f_{n_k}(x)$. This sum is displayed for $N=1000$ in Fig.~\ref{fig:introduction_2} on the left. \\
Our example is known as a \emph{lacunary Fourier series}, that is, its frequencies fulfill the growth condition given above. Its random properties are a classical field of research. In the literature, the sequence of prime numbers $(2\pi p_k)_{k\geq 0}$ is often cited as a counterexample for a sequence of frequencies which does not give rise to a lacunary Fourier series: it does neither fulfill the growth condition nor alternative conditions on arithmetic patterns which exist in the literature. However, experiments in this article suggest that $\sum_k \sin(\pi p_k x)$ share a lot of the random properties of lacunary series (see Fig.~\ref{fig:introduction_2} for a first impression or \cite{Vartziotis_Wipper_2016}): for instance, the central limit theorem seems to hold. Unfortunately, this looks difficult to prove.\\
\begin{figure}
\begin{minipage}{0.46\textwidth}
 \includegraphics[width=\textwidth]{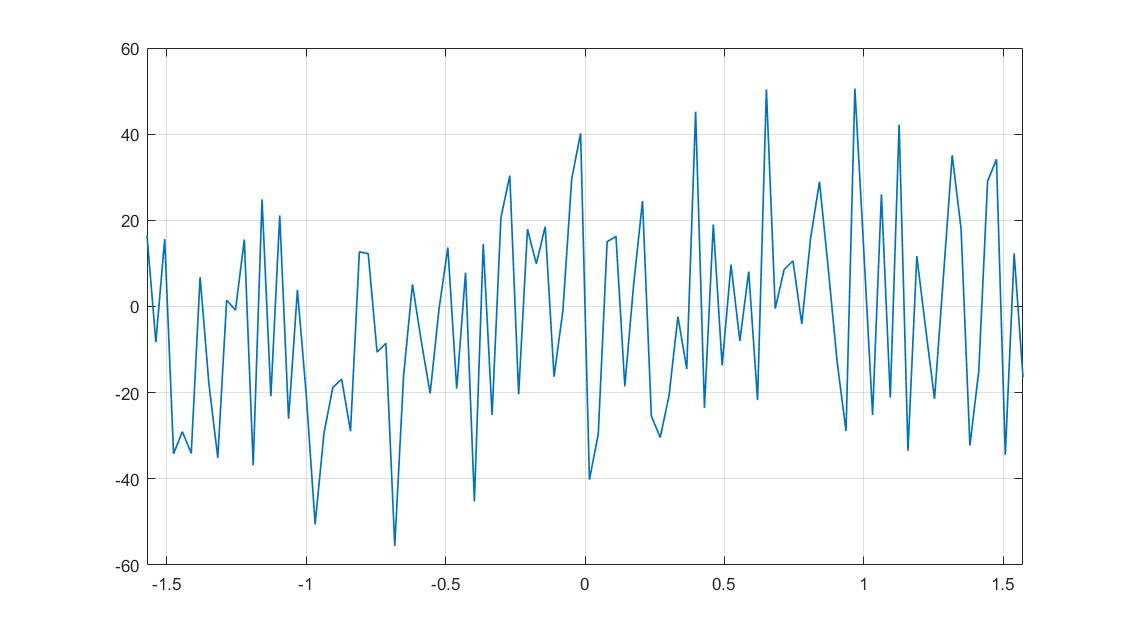}
\end{minipage} 
\begin{minipage}{0.46\textwidth}
 \includegraphics[width=\textwidth]{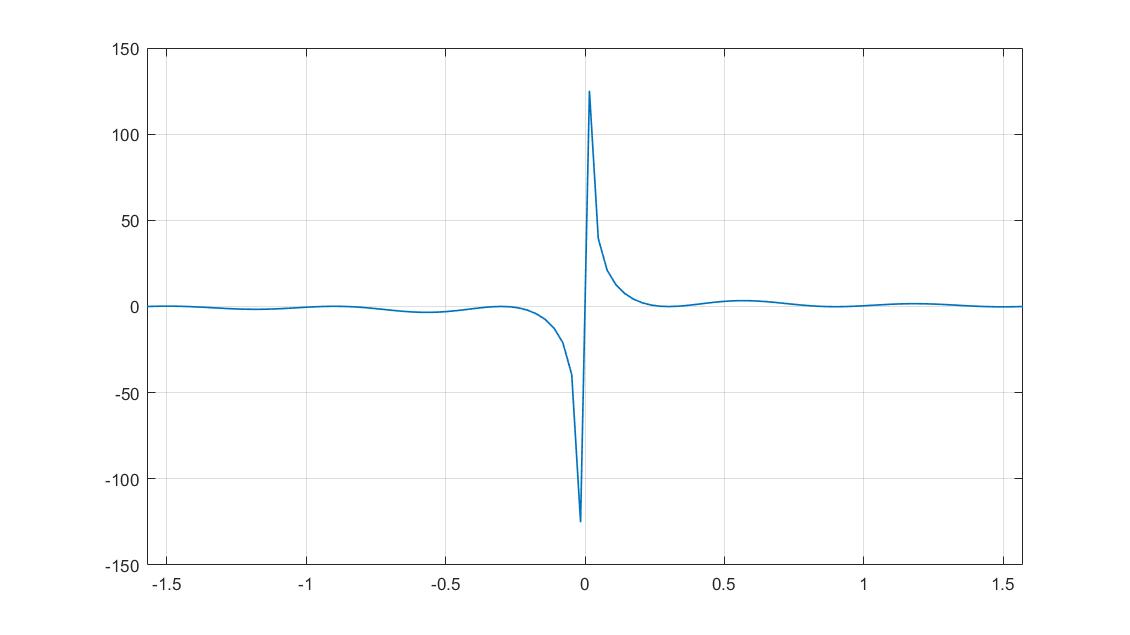}
\end{minipage}  
\begin{minipage}{0.46\textwidth}
 \includegraphics[width=\textwidth]{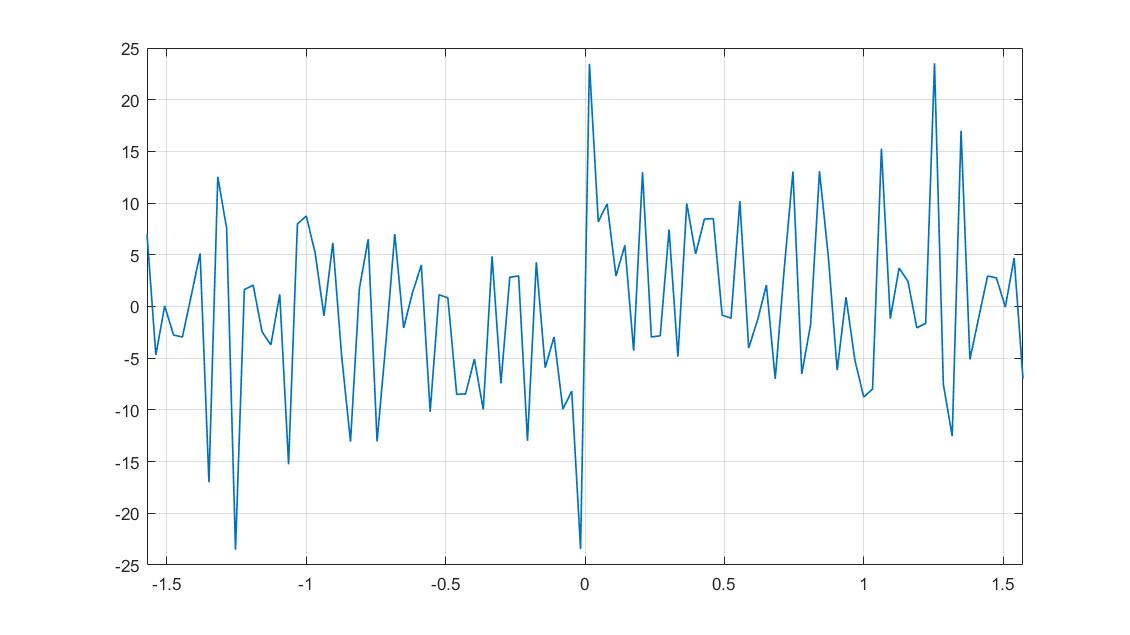}
\end{minipage}  
\caption{Graph of $\sum_{n=1}^{1000}\sin(2^n\pi x)$ on the left, of $\sum_{n=1}^{1000}\sin(n\pi x)$ in the middle and of $\sum_{p \leq 1000}\sin(p\pi x)$ on the right. } 
\label{fig:introduction_2}
\end{figure}
On the other hand, we can look at other manifestations of randomness in lacunary series (e.g. in the example in Fig.~\ref{fig:introduction_2}) and try to see if they are also present in our prime series $V_{\alpha,\beta}$: by introducing appropriate coefficients $a_k$, the walk $\sum_ka_k\sin(\pi n_kx)$ can be approximated by a \emph{Wiener process} which is an almost everywhere continuous random walk with independent normally distributed increments (see \cite{philipp_stout_1975} and \S~\ref{sec:random}). This implies directly a lot of interesting properties for the series, e.g. the law of iterated logarithm holds. It would be interesting if a similar approximation exists for our series $V_{\alpha,\beta}$. Again, we were not able to prove this.\\

But we can show that our series $V_{\alpha,\beta}$ has in fact for specific $\alpha,\beta$ properties in common with a Wiener process, e.g. its regularity and fractality (see \S~\ref{sec:differentiability}).  
The above-mentioned example $\sum_k a_k \sin(2^k\pi x)$ is in fact famous for these reasons: It belongs to the family of Weierstrass functions $F_{a,b}(x) = \sum_{n=0}^{\infty}a^n\sin(b^nt)$ which have been extensively studied for its differentiability properties. Under certain conditions on $a,b$ this function is nowhere differentiable, but H\"older continuous. \\
Another historical example which is non-differentiable, but multifractal, is the Riemann function $R_{2}(x)=\sum_{n=1}^{\inf}n^{-2}\sin(n^2x)$. Note, that it is not a lacunary series as $\sfrac{(n+1)^2}{n^2} \rightarrow 1$. With our prime series, we place ourselves in between these two historical examples with respect to the growth of its frequencies. \\

While prime sums are extensively studied in the context of the famous prime conjectures (e.g. for Vinogradov's theorem and the like), we have not found a treatment of trigonometric series over prime numbers. The reason for this is most probably that these series have not the necessary form to help to progress in the proofs of the prime conjectures where prime exponential sums play a dominant role. As mentioned above, these series have not been studied in the context of lacunary series as prime numbers do neither grow fast enough nor do they have known arithmetic properties which are necessary for a straightforward analysis.\\
By using results of prime number theory, we are nevertheless able to show conditions on the differentiability and self similarity of our prime series. Experimentally, we explore also its box dimension in dependence on $\alpha,\beta$. 
\begin{rem}
For most of our questions, we can restrict ourselves -- without loss of generality -- to the real part $\sum_p p^{-\alpha}\cos(2\pi p^{\beta}t)$ of the series which we denote by $V_{\alpha,\beta}(t)$, too.
\end{rem}

\section*{Acknowledgement}
We would like to thank the referee for his remarks and Florian Pausinger for his helpful hints and questions. 
\section{Convergence and differentiability}\label{sec:differentiability}

There are basically two factors which influence the smoothness and convergence of a function series $\sum_{k}a_k\exp(2\pi in_k t)$ as ours: 
\begin{enumerate}
\item The faster the coefficients $a_k$ decrease for $k \rightarrow \infty$, the smaller is the influence of the higher frequencies. This implies that the series converges better and the resulting function is smoother.
\item The faster the frequencies $n_k$ increase or equivalently, the greater the gaps, the smaller gets the period of the oscillation so that one obtains more peaks and sinks in one interval which increases the fractal character. 
\end{enumerate}
\subsection{Historical remarks }\label{sec:historical}
The nature of these influences, easily deduced, are also backed by the long history of studies on the following two families of functions (and derived families): \\
Let 
$$F_{a,b}(t)=\sum_{n=0}^{\infty}a^n\cos(b^nt)$$
be the family of Weierstrass functions which have been extensively studied. One knows the following: 
\begin{theorem}[\cite{hardy1916},\cite{jaffard2010}]
\begin{enumerate}
\item If $0 < ab <1, a<1,b>1$, then $F_{a,b}$ is differentiable. 
\item If $0<a<1<ab$, then $F_{a,b}$ is nowhere differentiable. Further, the H\"older exponent is a constant function $s = -\frac{\log a}{\log b}$, i.e. for all $t,t_0$ it holds
$$\left|F_{a,b}(t)-F_{a,b}(t_0)\right| \leq C\left|t-t_0\right|^{s}.$$ 
\end{enumerate}
\end{theorem}
On the other hand, one has the family of Riemann's function (whose authorship by Riemann is apparently only confirmed by Weierstrass) defined by 
$$R_{\alpha}(x)=\sum_{n=1}^{\infty} n^{-\alpha}\cos(n^2x)$$
which has the following proven properties:
\begin{theorem}[\cite{hardy1916}, \cite{hardy_littlewood1912},\newline \cite{gerver1970,gerver1970b},\cite{CC96}]
\begin{enumerate}
\item If $0< \alpha \leq \frac{1}{2}$, then the series is not a Fourier series of a $L^1$-function. If $0 < \alpha < \frac{1}{2}$, then $R_{\alpha}$ converges at $x$ if and only if $x =\frac{a}{q}$, where $a,q$ are coprime and $4$ divides $q-2$. 
\item If $\frac{1}{2} < \alpha < 1$, then the series converges in $p$-norm to a $L^p$-function for $p < \frac{2}{1-\alpha}$. 
\item If $\alpha =1$, then the series has bounded mean oscillation. 
\item If $\alpha < \frac{5}{2}$, then $R_{\alpha}$ is not differentiable at any irrational value of $x$, and its Hausdorff dimension for $\frac{3}{2}\leq \alpha \leq \frac{5}{2}$ is equal to $$\dim_H(R_{\alpha})=\frac{9}{4}-\frac{\alpha}{2}.$$ If $\alpha=2$, then $R_2$ is differentiable at $x$ if and only if $x=\frac{a}{q}$ where $a,q$ are coprime and $4$ divides $q-2$. 
\item If $\alpha=2$, the Hoelder exponent is discontinuous everywhere. In fact, $R_{2}$ is a function with unbounded variation and multifractal. 
\end{enumerate}
\end{theorem}
In the following, we aim to give a similar description for our function series. Let us start with some preliminary definitions which are necessary for what follows:
\subsection{Preliminary definitions}
We call a function $f: \mathbb{R} \rightarrow \mathbb{C}$ \emph{locally H\"{o}lder continuous} at $x_0 \in \mathbb{R}$, if there exist $s \in (0,1]$ and $C,\epsilon >0$ such that 
$$\left|f(x)-f(y)\right|\leq C\left|x-y\right|^{s}, \quad \mbox{for all}\;x,y \in B_{\epsilon}(x_0).$$
We call the supremum of $s$ for which these inequality holds at $x_0$ the \emph{local H\"{o}lder exponent}. \\
Let $\phi: \mathbb{R} \rightarrow \mathbb{C}$ be a smooth function with compact support $\supp(\phi) \subset \mathbb{C}$. We write 
$$\hat{\phi}(u)=\int_{\mathbb{R}} \phi(t)\exp(-iut) dt$$
for the \emph{Fourier transform} of $\phi$. Further, let $\phi: \mathbb{R} \rightarrow \mathbb{R}$ be given such that the support of its Fourier transform is contained in $[-1,1]$, then the \emph{Gabor wavelet transform} of a function $f: \mathbb{R} \rightarrow \mathbb{C}$ is defined by 
$$G(a,b,\lambda) = \frac{1}{a}\int_{\mathbb{R}}f(t)\exp(-i\lambda t)\phi\left(\frac{t-b}{a}\right)dt.$$
With these notation, we have the following estimation which is a special case of Proposition 5 in \cite{jaffard2010}: 
\begin{prop}[Jaffard]\label{prop:jaffard}
Let $f:\mathbb{R}\rightarrow \mathbb{C}$ be a bounded function. Let $G(a,b,\lambda)$ be the Gabor wavelet transform of $f$. If $f$ is locally H\"{o}lder continuous at $x_0 \in \mathbb{R}$ with H\"older coefficient $s$, then there exists $C > 0$ such that for all $a \in (0,1]$ and for all $b \in \overline{B_{1}(x_0)}$ and for all $\lambda \geq a^{-1}$ we have
$$\left|G(a,b,\lambda)\right| \leq Ca^{s}\left(1 + \frac{\left|x_0 -b\right|}{a}\right)^{s}.$$
\end{prop} 
  
\subsection{Differentiability of $V_{\alpha,\beta}$}
In the spirit of the results in \S~\ref{sec:historical}, we aim to determine which conditions have to be fulfilled by the coefficients and frequencies in our example in order to have a certain degree of differentiability.
Firstly, we consider 
$$V_{\beta}(n,t)=\sum_{p \leq n}f(p) \cos(2\pi p^{\beta} t), \quad \beta >0,$$
where $f$ is any function of prime numbers. We can state the trivial fact that 
\begin{prop}
For any $\beta \geq 0$, if $\int_{2}^{\infty}\frac{\left|f(x)\right|}{\ln(x)}dx < \infty$, then the partial sums $V_{\beta}(n,t)$ converges uniformly and absolutely to a continuous function denoted by $V_{\beta}$. 
\end{prop}
\begin{proof}
We have $\left|f(p) \cos(2\pi p^{\beta} t)\right| \leq \left|f(p)\right|$ for all $p$. By the Weierstrass $M$-test the partial sums $V_{\beta}(n,t)$ converges uniformly and absolutely if $\sum_{p}\left|f(p)\right| < \infty.$ Using the Riemann-Stieltjes Integral and the Prime number theorem we get
$$\sum_{p} f(p) = \int_{2}^{\infty}f(x)d\pi(x) = \int_2^{\infty}\frac{f(x)}{\ln(x)}dx,$$ 
where $\pi(x)$ denotes the number of primes $\leq x$ finishing the proof.
\end{proof}
We take now $$V_{\alpha,\beta}(n,t)=\sum_{p\leq n} p^{-\alpha}\cos(2\pi p^{\beta}t)$$ and denote with $V_{\alpha,\beta}(t)$ its limit whenever it exists. Then one can show the following statement:
\begin{theorem} Let $\alpha \in \mathbb{R}$ and $\alpha > 1$. 
\begin{enumerate}
\item Then the series $V_{\alpha,\beta}(n,t)$ converges uniformly and absolutely to a continuous function $V_{\alpha,\beta}(t)$. 
\item For $m \geq 1$, if further $\alpha -m\beta > 1$, then the function $V_{\alpha,\beta}(t)$ is $C^m$, i.e. $m$ times continuously differentiable.
\end{enumerate}
\end{theorem}
\begin{proof}
For the first result we use the properties of the prime zeta function $P(\alpha)=\sum_{p}p^{-\alpha}$: it converges absolutely for $\alpha >1, \alpha \in \mathbb{R},$ and diverges for $\alpha= 1$ (see e.g. \cite{landau_walfisz1920},\cite{froeberg1968}). The coefficients $p^{-\alpha}$ are an upper bound for the terms $p^{-\alpha}\cos\left(2\pi p^{\beta}t\right)$. Consequently, the Weierstrass $M$-test implies that for $\alpha > 1$ and any $t \in [0,1)$, $V_{\alpha,\beta}(n,t)$ converges uniformly and absolutely to $V_{\alpha,\beta}(t)$. As any partial sum $V_{\alpha,\beta}(n,t)$ is continuous, the limit is a continuous function, too. \\
Secondly, for any $n$ and $t$ we can differentiate the partial sums 
\begin{align*}
V'_{\alpha,\beta}(n,t)&=-2\pi \sum_{p\leq n} p^{-\alpha+\beta}\sin(2\pi p^{\beta} t). 
\end{align*} 
This sequence of derivatives converges uniformly with the same argument as above for $\alpha-\beta> 1$, so that one concludes that $V_{\alpha,\beta}(t)$ is continuously differentiable itself with derivative $V'_{\alpha,\beta}(t)=-2\pi \sum_{p} p^{-\alpha+\beta}\sin(2\pi p^{\beta} t)$. By induction over $m$, one proves the $m$-time differentiability of the function. 
\end{proof}
\begin{rem}
The result is in accordance with the intuitive smoothness of the series: for fixed $\alpha > 1$, the series gets the smoother, the smaller the frequency $p^{\beta}$, $\beta \rightarrow 0$, or equivalently, the larger the period. Therefore, the peaks and sinks of the oscillation are more and more separated so that the series gets smoother (see Fig.~\ref{fig:regularity_1}-\ref{fig:regularity_3}).
\end{rem}

\begin{figure}
\includegraphics[width=\textwidth]{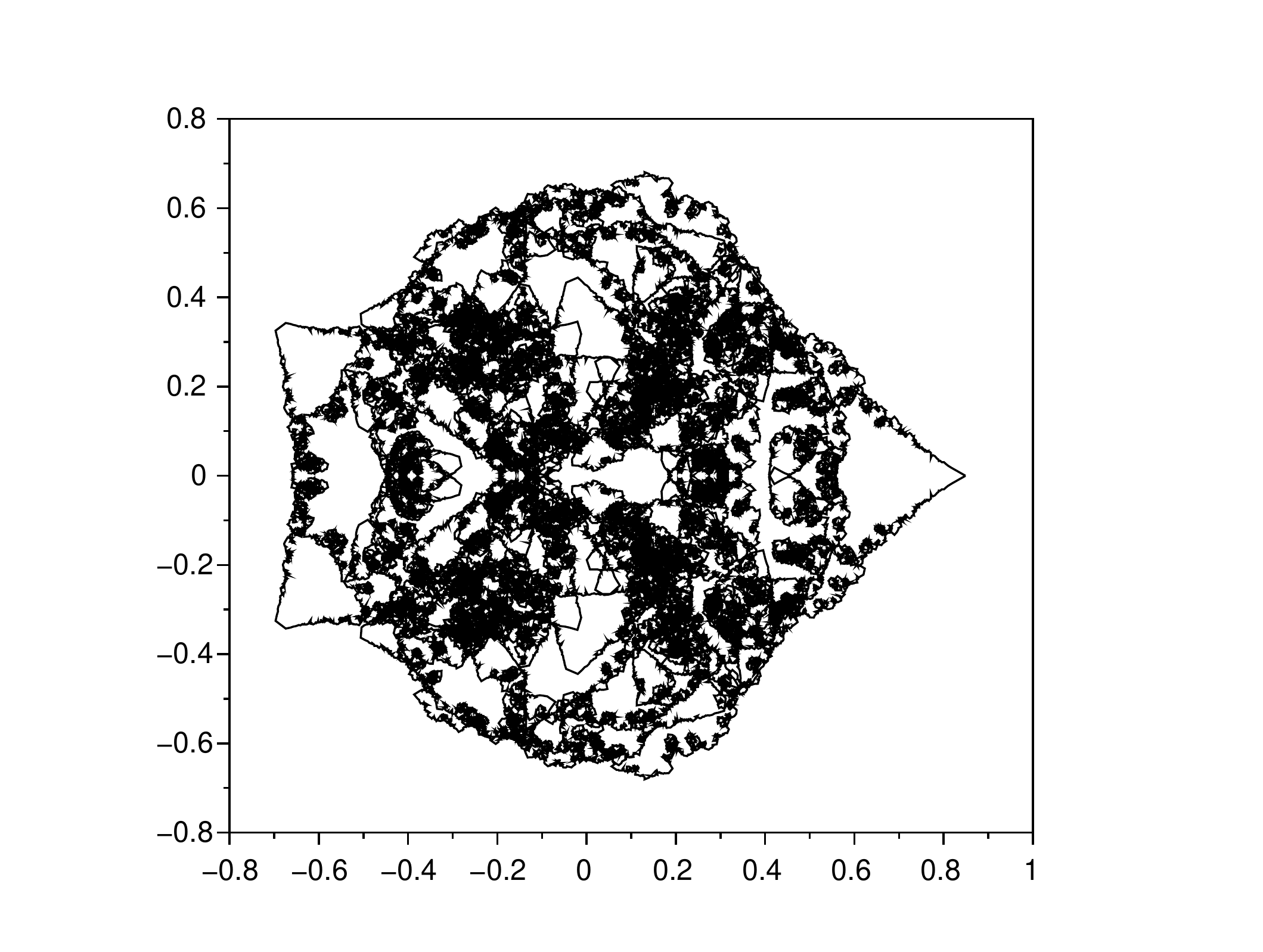}
\caption{Graph of $V_{1.5,2}(10^5,t)$ at $5*10^4$ discrete points in each direction (interpolated). } 
\label{fig:regularity_1}
\end{figure}

\begin{figure}
\includegraphics[width=\textwidth]{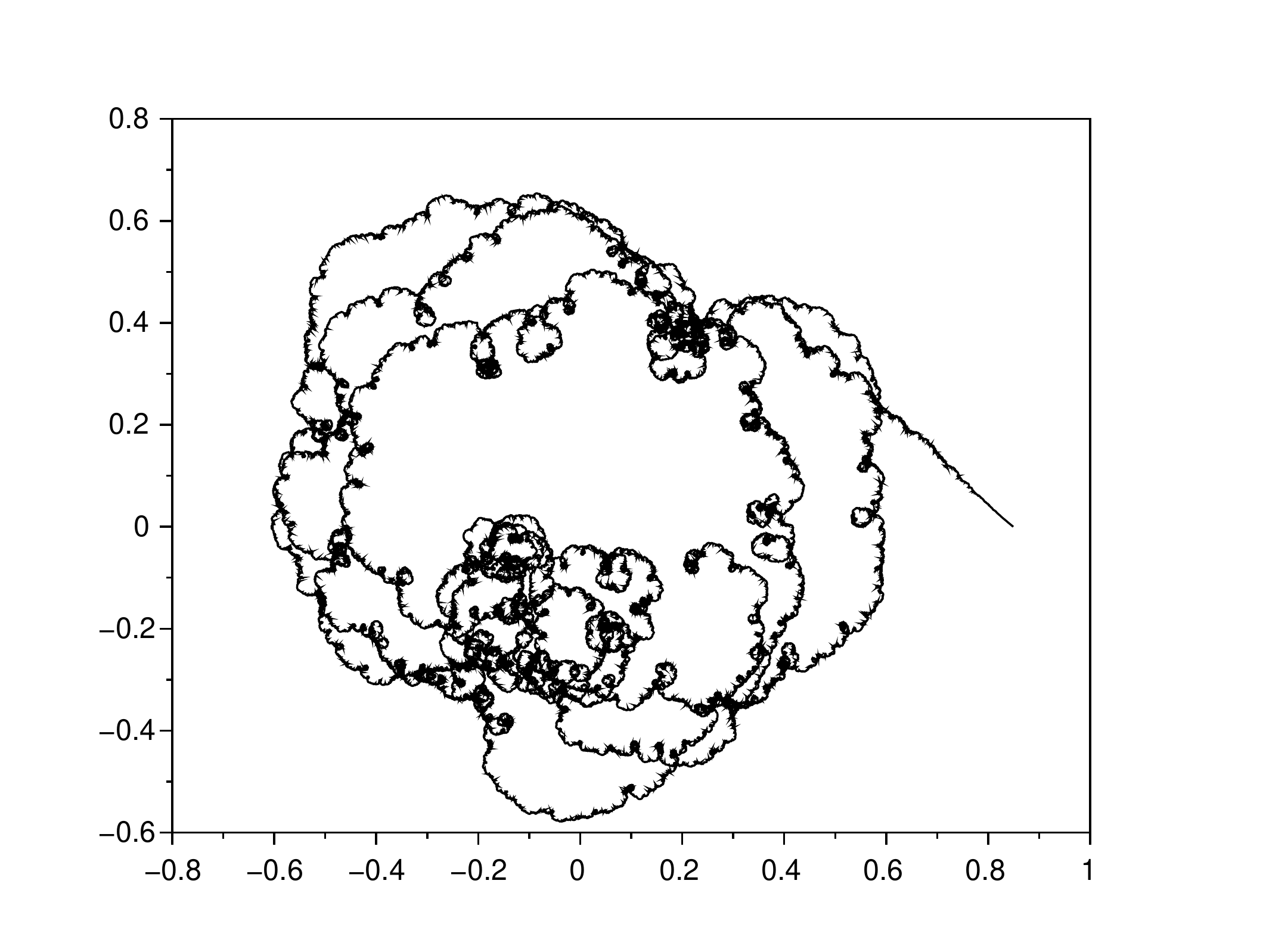}
\caption{Graph of $V_{1.5,1.5}(10^5,t)$ at $5*10^4$ discrete points in each direction (interpolated).} 
\label{fig:regularity_2}
\end{figure}

\begin{figure}
\includegraphics[width=\textwidth]{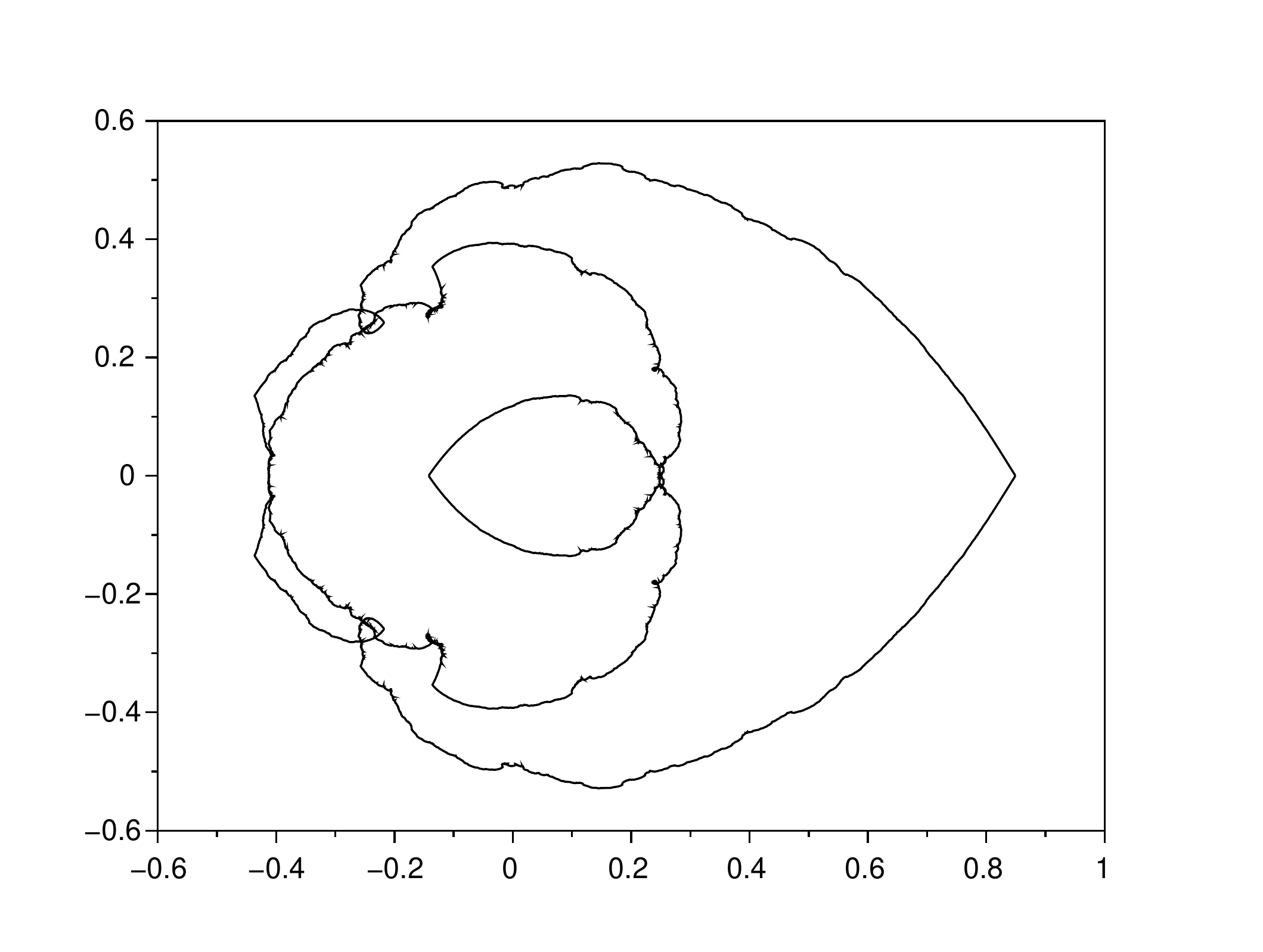}
\caption{Graph of $V_{1.5,1}(10^5,t)$ at $5*10^4$ discrete points in each direction (interpolated).} 
\label{fig:regularity_3}
\end{figure}

\begin{theorem}
If $1 < \alpha \leq \beta +1$, then the function is H\"{o}lder continuous with H\"{o}lder coefficient $ s\leq \frac{\alpha}{\beta}$.
\end{theorem}
\begin{proof}
First of all, let $f:\mathbb{R} \rightarrow \mathbb{C}$ be an integrable function and $N >0$, then by using the Riemann-Stieltjes integral (see e.g. \cite{rosser1962}) and the Prime number theorem as above one knows
$$\sum_{p \leq N} p^{-\alpha} = \int_2^{N}x^{-\alpha}d\pi(x) \sim \int_2^N \frac{1}{x^{\alpha}\ln(x)}dx.$$
From this formula and $\li(x)$ denoting the logarithmic integral function, one deduces (substituting $dx$ by $d\left(x^{1-\alpha}\right)$) for $\alpha < 1$ 
$$\sum_{p\leq N}p^{-\alpha} = (1-\alpha)^{-1}\li\left(N^{1-\alpha}\right) + \mathcal{O}\left(N^{1-\alpha}\exp(-c\sqrt{\ln(N)}\right).$$
Approximating the logarithmic integral this implies \begin{equation}\label{eq:first_term}\sum_{p \leq N} p^{-\alpha} \sim \frac{N^{1-\alpha}}{(1-\alpha)\ln(N)}.\end{equation}
If $\alpha > 1$, we have to use the explicit formula for the prime zeta function to get an estimate for the speed of convergence (see e.g. \cite{cohen2000} for a derivation of the formula). We then have by partial summation 
\begin{align*}
\sum_{p}p^{-\alpha} &= \sum_{p \leq N} p^{-\alpha} + \sum_{n=1}^{\infty}\frac{\mu(n)}{n}\ln\left(\zeta(N,\alpha n)\right), \quad \mbox{with}\\
\zeta(N,\alpha)&=\zeta(\alpha)\Pi_{p \leq N}\left(1-p^{-\alpha}\right),\\
\end{align*}
where $\zeta(\alpha)=\sum_{n=1}^{\infty}n^{-\alpha}$ denotes the \emph{Riemann zeta function} and $\mu$ the \emph{Moebius function}. So we get for the tail of the prime zeta function 
\begin{align}\label{eq:second_term}
\sum_{p> N}p^{-\alpha} &= \sum_{n=1}^{\infty}\frac{\mu(n)}{n}\ln\left(\zeta(N,\alpha n)\right) \quad \mbox{with}\\
\ln\left(\zeta(N,\alpha)\right) &= \mathcal{O}\left(N^{-\alpha}\right) \quad \mbox{and}\nonumber\\
\sum_{n=1}^{\infty}\frac{\mu(n)}{n} &= 0. \nonumber
\end{align}
Combining Eq.~(\ref{eq:first_term})-(\ref{eq:second_term}) on the asymptotic of the prime zeta function, we can estimate now the regularity of our function $V_{\alpha,\beta}(t)$. \\
For any $t,t_0 \in [0,1)$, we choose $N = \left|t-t_0\right|^{-\frac{1}{\alpha}}.$ Then we have with the mean value theorem and using the absolute convergence of the series
\begin{align*}
\left|V_{\alpha,\beta}(t)-V_{\alpha,\beta}(t_0)\right| &\leq \sum_{p \leq N} p^{-\alpha}\left|\cos(2\pi p^{\beta}t)-\cos(2\pi p^{\beta}t_0)\right| + 2\sum_{p > N} p^{-\alpha}\\
&\leq \sum_{p \leq N} p^{-\alpha + \beta}\left|t-t_0\right| + 2\sum_{p > N} p^{-\alpha}\\
&\leq \frac{N^{-\alpha + \beta +1}}{(\beta-\alpha +1)\ln(N)} \left|t - t_0\right| + 2C N^{-\alpha}\\
&\leq C \left|t-t_0\right|^{2-\frac{\beta+1}{\alpha}}. 
\end{align*}
The exponent $1-\beta < 2-\frac{\beta +1}{\alpha}\leq 1$ is not necessarily optimal, but a lower bound. But it suffices to conclude that the function is H\"older continuous so that we can derive an upper bound for its H\"older exponent: \\For this step, we use a method developed by Jaffard in \cite{jaffard2010} which relies on a wavelet transform and the idea to choose the wavelet transform such that only one frequency of $V_{\alpha,\beta}(t)$ is picked up. Let $\theta_m= \min\left\{p_m^{\beta}-p_{m-1}^{\beta},p_{m+1}^{\beta}-p_m^{\beta}\right\}$ and $\Delta_m=p_m-p_{m-1}$. \\
We choose a function $\phi$ whose Fourier transform $\hat{\phi}$ has compact support $\supp(\hat{\phi}) \subset [0,1]$ and $\hat{\phi}(0)=1$. We then look at the Gabor-wavelet transform
\begin{align*}
G_m(\theta_m^{-1},t_0,p_m^{\beta})&=\theta_m \sum_{k}p_k^{-\alpha}\int_{\mathbb{R}}\exp\left(i\left(p_k^{\beta}-p_m^{\beta}\right)t\right)\phi\left(\theta_m(t-t_0)\right)dt\\
&=\sum_k p_k^{-\alpha} \exp\left(i\left(p_k^{\beta}-p_m^{\beta}\right)t_0\right) \int_{\mathbb{R}}\exp\left(i\frac{\left(p_k^{\beta}-p_m^{\beta}\right)}{\theta_m}\right)\phi(u)du, 
\end{align*}
with $u=\theta_m(t-t_0)$. Substituting $\hat{\phi}(y)= \int_{\mathbb{R}}\exp(iyu)\phi(u)du$ for $y=\frac{\left(p_k^{\beta}-p_m^{\beta}\right)u}{\theta_m}$ in the equation we get
$$G_m(\theta_m^{-1},t_0,p_m^{\beta}) = \sum_k p_k^{-\alpha}\exp\left(i\left(p_k^{\beta}-p_m^{\beta}\right)t_0\right)\hat{\phi}\left( \frac{\left(p_k^{\beta}-p_m^{\beta}\right)u}{\theta_m}\right).$$
As the support of $\hat{\phi}$ is a subset of the unit interval, it does vanish for any $k \neq m$, so the expression it reduced to
\begin{equation}\label{eq:hoelder}
G_m(\theta_m^{-1},t_0,p_m^{\beta})= p_m^{-\alpha}
\end{equation}
Recall that we have just proved that $V_{\alpha,\beta}$ is locally H\"{o}lder continuous at $t_0 \in \mathbb{R}$. Further, for all $m$ it is $p_m^{\beta} \geq \theta_m$ and $\theta_m^{-1} \in (0,1]$. Hence, applying Proposition~\ref{prop:jaffard}, there exists $C > 0$ such that for all $s \in (0,1)$
$$G_m(\theta_m^{-1},t_0,p_m^{\beta}) =p_m^{-\alpha}\leq C\theta_m^{-s}.$$
The gap $\theta_m$ is bounded by $p_m^{\beta}$ from above so that the H\"{o}lder coefficient $s$ is bounded by $\frac{\alpha}{\beta}$ from above finishing the proof.
\end{proof}
\begin{rem}
Let $\alpha>1$ be fixed. The bigger the gaps of the frequency, $\beta \rightarrow \infty$, the stronger the irregularity of $V_{\alpha,\beta}(t)$.
\end{rem}

\subsection{Self similarity and fractal dimension}
The graph of the function $V_{\alpha,\beta}$ seems to be self similar for certain $\alpha,\beta$. There seems to be an approximate scalar invariance a points $q^{-1}$, where $q$ is prime. Let us make more precise this intuition: look for example on the partial sums $V_{1,1}(n,t)=\sum_{p\leq n}p^{-1}\exp(2\pi i pt)$ in Fig.~\ref{fig:oneone}:
\begin{figure}
\includegraphics[width=\textwidth]{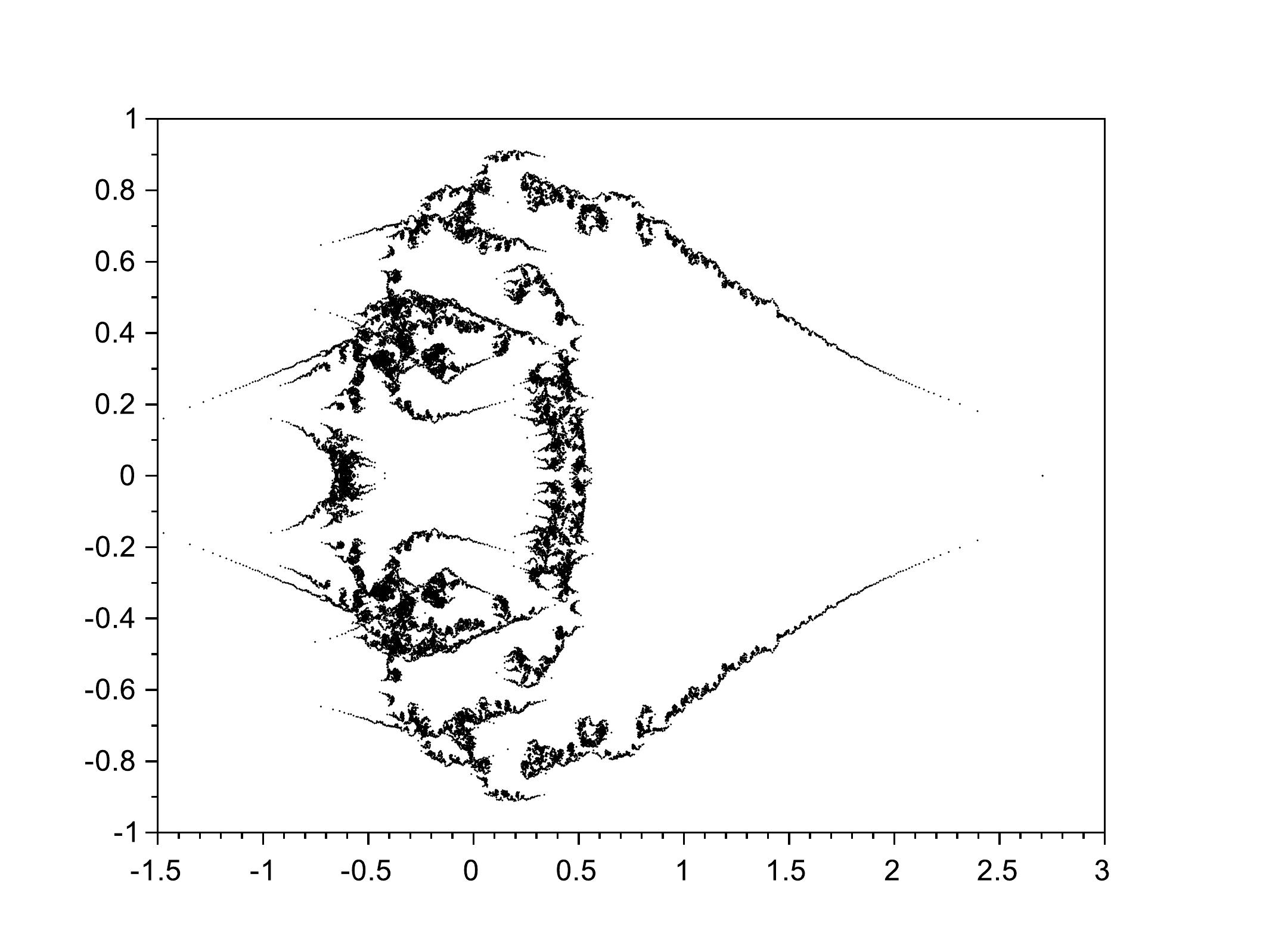}
\caption{Graph of $V_{1,1}(10^5,t)$ at $5*10^4$ discrete points.}
\label{fig:oneone}
\end{figure} 
Denote by $p_k$ the $k$th prime number. We restrict ourselves again to the real part of $V_{1,1}(n,t)$. The point $\frac{1}{2}$ is a global minimum as $V'_{1,1}\left(n,\frac{1}{2}\right)=0$ and $V_{1,1}\left(n,\frac{1}{2}\right)=\frac{1}{2}-\sum_{k=1}^np_k^{-1}$ as the primes greater than $2$ are odd. Now, consider the point $\frac{1}{3}$: we have $V_{1,1}\left(n,\frac{1}{3}\right)=\frac{1}{3} -\frac{1}{2}\sum_{k=1,k\neq 2}^np_k^{-1}$. More generally, one has
\begin{align*}
V_{1,1}\left(n,\frac{1}{q}\right)&=\frac{1}{q} + \sum_{l=1}^{q-1}\left(\cos\left(\frac{2\pi l}{q}\right)\sum_{p_k=l\mod q}^np_k^{-1}\right), \quad q\;\mbox{prime}\\
&= \sum_{l=0}^{q-1}\cos\left(\frac{2\pi l}{q}\right) R_{l,q}\\
\end{align*}
That is, we can decompose the partial sum into residue classes of the prime numbers and the roots of unity of cosine. 
One knows that the number of primes $p \leq n$ that are congruent to $ l \mod q$ are approximately the same for all $l$, that is, $\frac{n}{\Phi(q)\log(n)}$ where $\Phi(q)$ denotes the Euler totient function and is equal to $q-1$ for $q$ prime. So for any $\frac{1}{q}$, $q$ prime, one can use this distribution and the Riemann-Stieltjes integral to show that the difference between the sums $\sum_{p_k=l\mod q,\,p_k\leq n}p_k^{-1}$ for each $l=1,\dots,q-1$ converges to zero for $n \rightarrow \infty$, that is:
\begin{align*}
R_{l,q} &= \sum_{p_k =l\mod q,\,p_k\leq n}p_k^{-1} \sim\frac{1}{q-1}\int_2^n \frac{1}{x\ln(x)}dx\\
&=\frac{1}{q-1}\left(\ln \ln(n) + C\right).
\end{align*}
The factors $\cos\left(\frac{2\pi l}{q}\right)$ are exactly the prime roots of unity and the sum $\sum_{l=0}^{q-1}\cos\left(\frac{2\pi l}{q}\right) = 0$. Consequently, one computes
\begin{align*}
V_{1,1}\left(n,\frac{1}{q}\right) &\sim \frac{1}{q} - \frac{1}{q-1}\left(\ln\ln(n) +C\right).
\end{align*}
As we have $V_{1,1}(n,1)=\sum_{p\leq n}p^{-1} \sim \ln\ln(n) + M$, one could argue that 
$$V_{1,1}(n,\frac{t}{q}) \approx \frac{1}{1-q}V_{1,1}(n,t)+\frac{1}{q}, \quad q \geq 3,\;\mbox{prime},$$
see Fig.~\ref{fig:similarity}. But keep in mind that these are only asymptotic equivalences while our partial sum $V_{1,1}(n,t)$ do not converge for $n \rightarrow \infty$, so the self similarity of the graph is certainly not strict.
 \begin{figure}
\includegraphics[width=\textwidth]{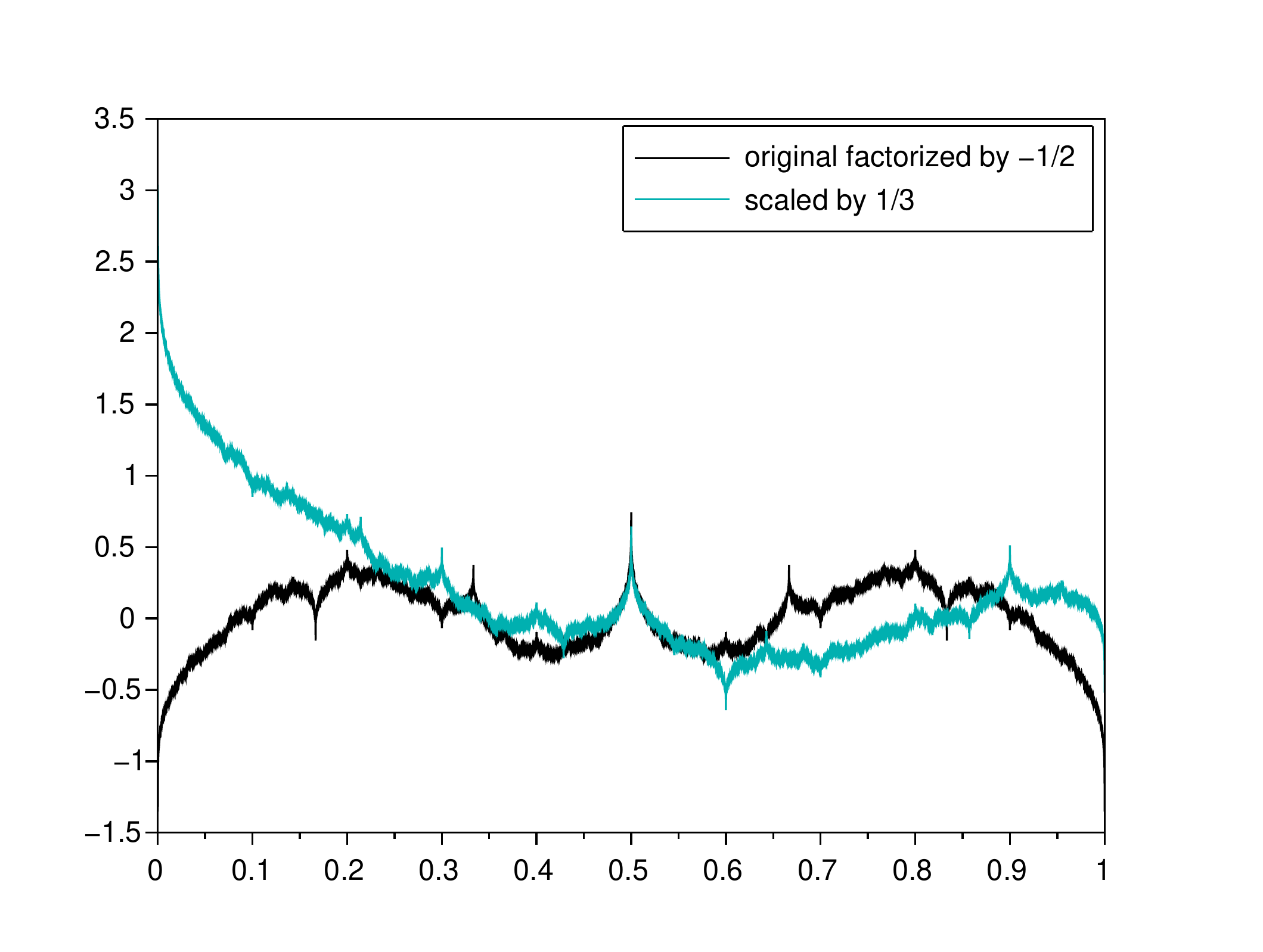}
\caption{The graph of the real part of $-\frac{1}{2}V_{1,1}(10^6,t)$ (in black) and $V_{1,1}(10^6,t/3)+\frac{1}{3}$ (in grey). }
\label{fig:similarity}
\end{figure} 
\subsubsection{Fractal dimension of $V_{\alpha,\beta}$.}
Further, we compute numerically the \textit{box dimension} of the graph of $V_{\alpha,\beta}$ defined in the following way: let $A:=[a,b]\times[c,d]$ be the rectangle such that the graph $V_{\alpha,\beta}(n,t) \subset A$ is contained. We compute then for $i,j=0,\dots N-1$ the intersections $V_{\alpha,\beta}(n,t) \cap [a + i(b-a)/N,a+(i+1)(b-a)/N] \times [c + j(d-c)/N,c+(j+1)(d-c)/N]$. We denote the number of non-empty intersections by $M(N)$. The box dimension is then given by  
$$\dim_B\left(V_{\alpha,\beta}(n,t)\right)=\lim_{N\rightarrow \infty}\frac{\ln(M(N))}{\ln(N)}.$$
In accordance to our results on regularity of $V_{\alpha,\beta}$ we obtain the following Fig.~\ref{fig:box_dimension} for the (numerically computed) box dimension $\dim_B$ over the fraction $\frac{\alpha}{\beta}$. For $\alpha > 1$ fixed and $\beta \rightarrow \infty$, that is, $\frac{\alpha}{\beta}\rightarrow 0$, we expect that the fractal dimension converges to $2$. On the other hand, for $\beta \rightarrow 0$, the fractal dimension should converge to $1$ as the graph gets continuously differentiable if $\frac{\alpha-1}{\beta} > 1 + \frac{1}{\beta}$.
\begin{figure}
\includegraphics[width=\textwidth]{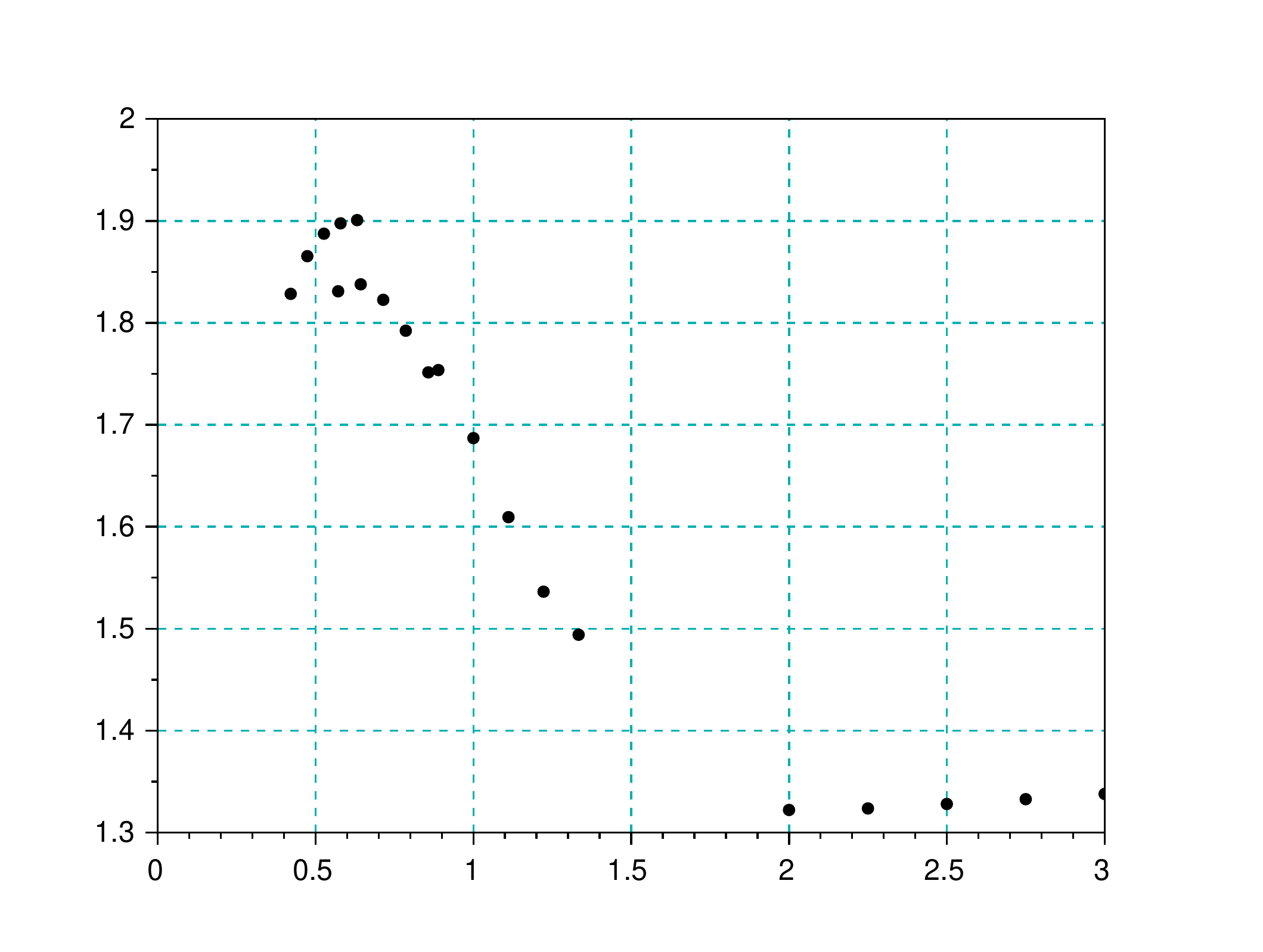}
\caption{Box dimension of $V_{\alpha,\beta}$ in dependence on the fraction of the powers $\frac{\alpha}{\beta}$ with $\alpha \in [1,1.5]$ and $\beta \in [0.5,3]$. Remark that $V_{\alpha,\beta}$ is not convergent for $\alpha=1$.}
\label{fig:box_dimension}
\end{figure}

\section{Random properties for $V_{\alpha,\beta}$}\label{sec:random}
The quite similar behavior of lacunary and random Fourier series let us think that it might be possible to capture the random character of the series $V_{\alpha,\beta}$ which is the subject of this section. Let us briefly review what it is known in the context of lacunary sequences and random variables:
\subsection{Lacunary sequences behaving as independent random variables: short overview}
The terms $(\sin(2\pi k x)_k$ and $(\cos(2\pi kx)_k$ behave like random variables, but strongly dependently. But if one restricts the sequence of frequencies $(2\pi k)_{k \geq 0}$ to $(2\pi n_k)_{k \geq 0}$ where the sequence $(n_k)_{k \geq0}$ has sufficiently fast growing gaps, i.e. \begin{equation}\label{eq:growth_condition}\frac{n_{k+1}}{n_k} \geq 1 + \rho, \quad \rho >0 \;\mbox{(Hadamard gap condition)},\end{equation}
then the sequences $\left(\sin(2\pi n_k x\right)$ behave like independent random variables. For example, one has 
$$\frac{1}{\sqrt{N}}\sum_{k=1}^N \sin(2\pi n_k x) \rightarrow \mathcal{N}(0,1),$$
where $\mathcal{N}(0,1)$ is the normal distribution. This was the main observation which has led to study the connections between lacunary and random Fourier series, most importantly the question which are the optimal growth conditions on the sequence $(n_k)_k$ such that the sequence $\left(f(n_k)\right)_{k \geq 0}$ for general periodic measurable functions $f$ with vanishing integral exhibits random properties (see historical overview \cite{kahane1997}). By introducing weights $a_k$ which obey certain growth conditions themselves, one can recover several limit theorems in complete analogy to random variables. In particular, the Central Limit Theorem (CLT) and the Law of Iterated Logarithm (LIL) are true (see results by Salem-Zygmund in \cite{salem_zygmund1947}, \cite{salem_zygmund1948}, Erd\"{o}s-G\'{a}l in \cite{erdoes_gal_1955} and Weiss in \cite{weiss_1959}). Further, it can be shown that the process can be approximated by a standard Brownian motion: 
\begin{theorem}[Philipp-Stout \cite{philipp_stout_1975}]
Assume the Hadamard gap condition. Assume further that $A_N:=\sqrt{\frac{1}{2}\sum_{k=1}^N a_k^2}\rightarrow \infty$ and there exists $\delta > 0$ such that $\lim_{N\rightarrow \infty} \frac{a_N}{A_N^{1-\delta}} =0$. Then without changing the distribution of the process $$S(t,x)=\sum_{k \leq t}a_k\cos(2\pi n_k x),\quad t \geq 0,$$it can be redefined on a suitable probability space together with a Wiener process $\left\{W(t)\;\big|\;t \geq 0\right\}$ such that 
$$S(t,x)=W\left(A_t\right) + \mathcal{O}\left(A_t^{\frac{1}{2}-\rho}\right), \quad \mbox{almost surely for some}\;\rho > 0.$$
\end{theorem}
While the Hadamard growth condition (\ref{eq:growth_condition}) for CLT can be weakened for general sequences $(n_k)_k$ (see \cite{erdoes_1962}) for coefficients $a_k=1$ to the optimal growth condition $\frac{n_{k+1}}{n_k} \geq 1 + \frac{c_k}{\sqrt{k}}$ with $c_k \rightarrow \infty$, one has observed that sequences with much slower growth can nevertheless satisfy the CLT if they fulfill certain arithmetic conditions, more precisely, bounds on the number of solutions for the diophantine equation. Results in this direction started with Gaposhkin (\cite{gaposhkin_1966}) and were recently sharpened by Berkes, Philipp and Tichy (\cite{berkes_philipp_tichy_2008}).  
The difficulties with the prime sequence $\left(p_k\right)_{k \geq 0}$ are on both sides: Firstly, while it is sure that the prime sequence is not a Hadamard sequence, neither precise lower nor upper bounds for the prime gap $p_{k+1}-p_k$ are known. The best results for a lower bound which would be of interest for us do not hold for all $k \geq 0$ but only infinitely many. For the upper bound it is proved by Goldston, Pintz and Yildirim (\cite{goldston_2009}) that $\lim \inf_{k \rightarrow \infty} \frac{\Delta_k}{\log p_k}=0.$
Secondly, there is no building law for prime numbers known and the infinite recurrence of certain patterns like twin primes are only conjectured but not completely proved. On the other hand, the random character of prime numbers is often invoked without being analytically established anywhere although the random model by Cram\'{e}r (\cite{cramer_1936}) is widely used and reproduces some results very efficiently (but fail in other aspects, e.g. in forecasting the size of the prime gap). In the question on convergence of functions $f(n_k)$ random models were also introduced (see e.g. \cite{schatte_1988}). Obviously, this is a broad and intensively studied mathematical subject where we do not dare to make contributions. Therefore, we stay more closely to our studied series: 
\subsection{The central limit theorem}
Because of the reason mentioned above we have not been able to show the central limit theorem for the random variables $\sin(\pi p_k x)$ or $\cos(\pi p_k x)$, the base of our series $V_{\alpha,\beta}$. Nevertheless, numerical computations strongly suggest that the central limit theorem holds, see Fig.~\ref{fig:clt}: we took $10^4$ uniformly distributed points $x$ of the interval $\left[-\frac{\pi}{2},\frac{\pi}{2}\right]$ and computed the sample average $\frac{1}{N}\sum_{k=1}^N\sin(p_k x)$ for $N=78498$, that is, the number of primes $\leq 10^6$. We computed the histogram for the values of the sample average which experimentally tends to a normal distribution as the size of the sample tends to infinity.  
\begin{figure}
\begin{minipage}{\textwidth}
 \includegraphics[width=\textwidth]{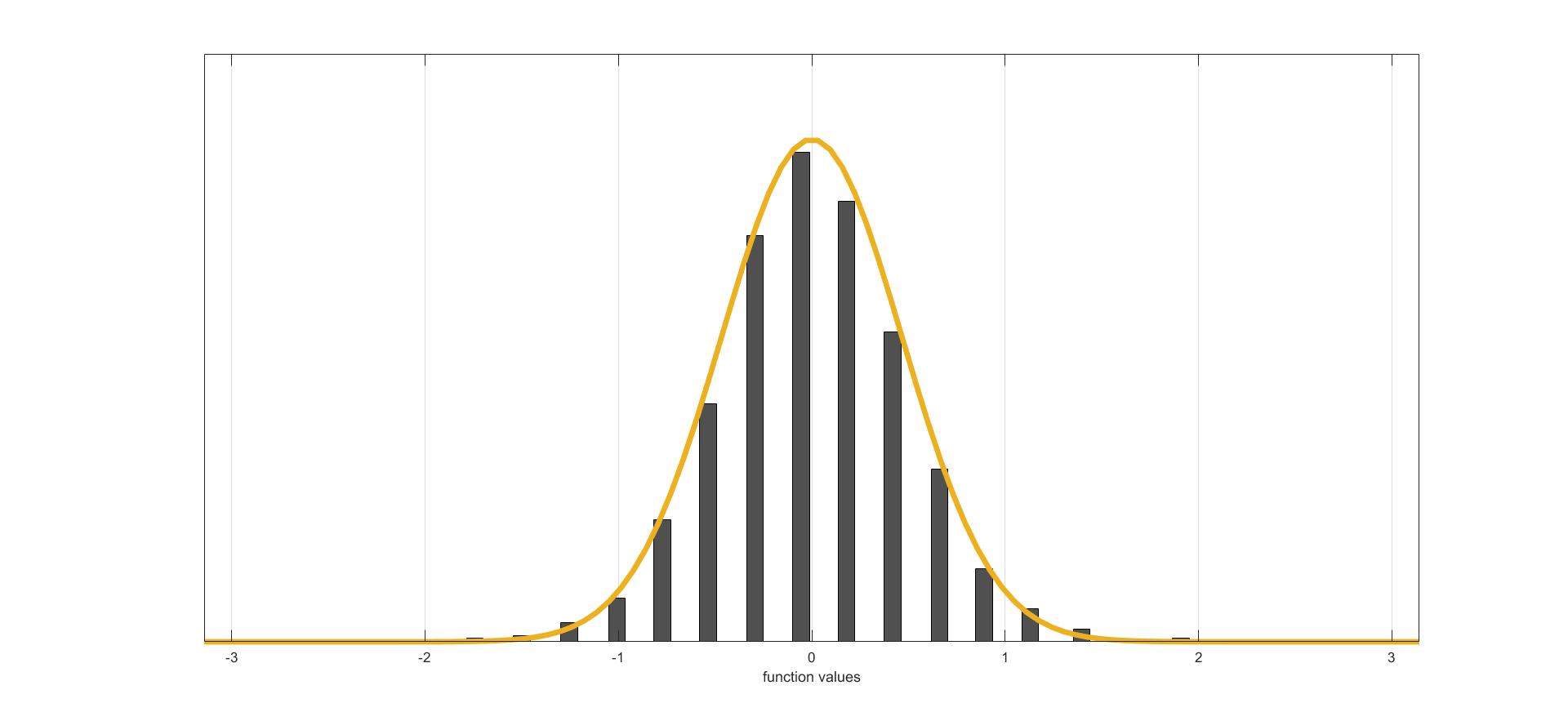}

\end{minipage}  
\begin{minipage}{\textwidth}
 \includegraphics[width=\textwidth]{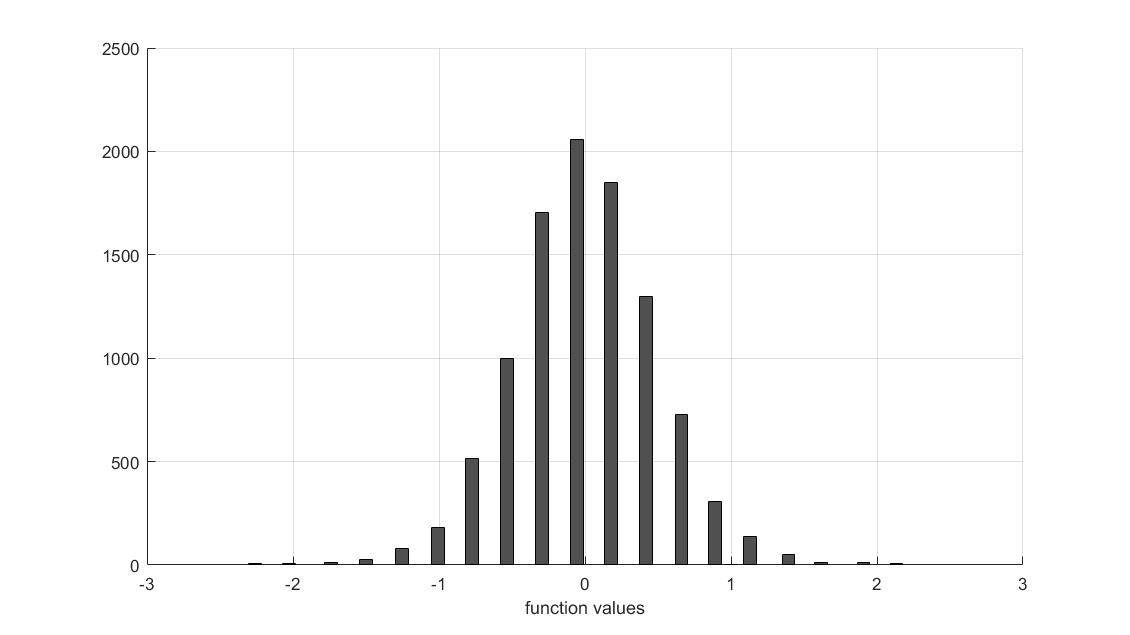}

\end{minipage} 
\caption{Normal distribution of $\frac{1}{N}\sum_{k=1}^N \sin(\pi p_k x)$ for $x$ uniformly distributed in $\left[-\frac{\pi}{2},\frac{\pi}{2}\right]$. } 
\label{fig:clt}
\end{figure}
\begin{figure}

 \includegraphics[width=\textwidth]{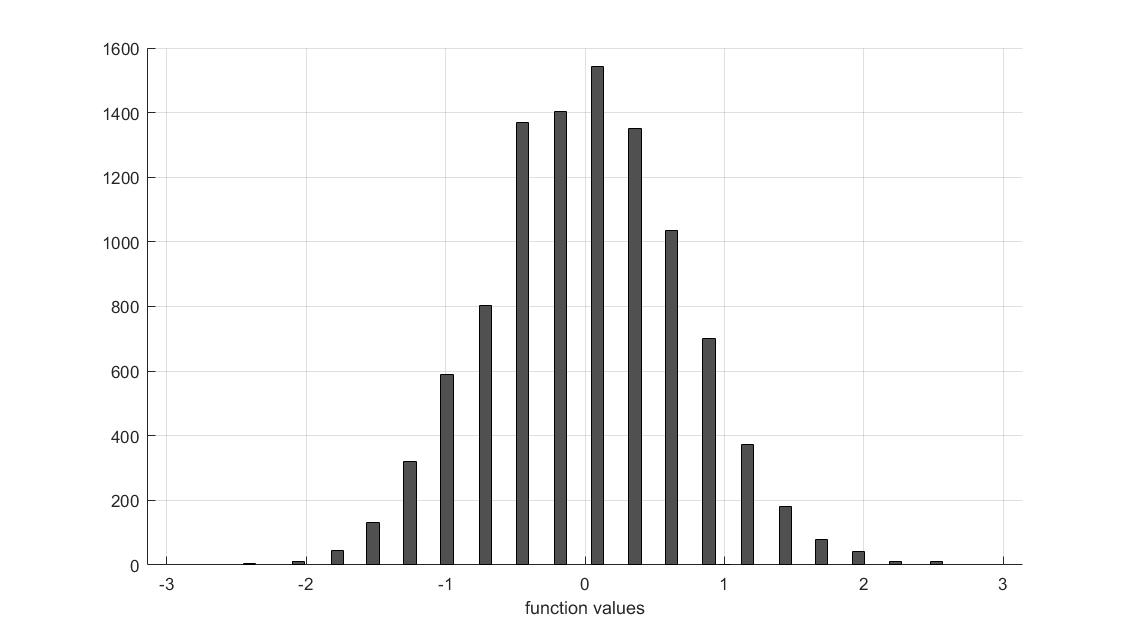}

\caption{Normal distribution of $\frac{1}{N}\sum_{k=1}^N \sin(\pi p_k^{\frac{3}{2}} x)$ for $x$ uniformly distributed in $\left[-\frac{\pi}{2},\frac{\pi}{2}\right]$. } 
\label{fig:clt2}
\end{figure}

\section{Concluding remarks}
The properties of the series $V_{\alpha,\beta}$ which we have discussed in this article are intimately related to the distribution of prime numbers, and it was mostly due to the unanswered questions on prime numbers that the analytical access to our series is limited. Therefore, knowledge on the distribution and bounds for the gaps of prime numbers would imply more or less directly properties where we were restricted to a numerical approach.

Although the series might remind of the Riemann zeta function or other number-theoretical functions, we did not construct $V_{\alpha,\beta}$ in this way and do not see a possibility to deduce it from any of them, besides from the trivial fact, that $V_{\alpha,\beta}(0)$ is equal to the prime zeta function $P(\alpha)=\sum_p p^{-\alpha}$.


\FloatBarrier
\bibliographystyle{apalike} 
\bibliography{bibtex}

\begin{thebibliography}{}

\bibitem[Berkes et~al., 2008]{berkes_philipp_tichy_2008}
Berkes, I., Philipp, W., and Tichy, R. (2008).
\newblock {\em Metric Discrepancy Results for Sequences {\{}nkx{\}} and
  Diophantine Equations}, pages 95--105.
\newblock Springer Vienna, Vienna.

\bibitem[Chamizo and C\'{o}rdoba, 1996]{CC96}
Chamizo, F. and C\'{o}rdoba, A. (1996).
\newblock Differentiability and dimension of some fractal fourier series.
\newblock {\em Advances in Mathematics}, 142:335--354.

\bibitem[Cohen, 2000]{cohen2000}
Cohen, H. (2000).
\newblock High precision computation of hardy-littlewood constants.
\newblock https://www.math.u-bordeaux.fr/~hecohen/.

\bibitem[Cram\'{e}r, 1936]{cramer_1936}
Cram\'{e}r, H. (1936).
\newblock On the order of magnitude of the difference between consecutive prime
  numbers.
\newblock {\em Acta Arith.}, 2:23--46.

\bibitem[D.~Goldston and Yildirim, 2009]{goldston_2009}
D.~Goldston, J.~P. and Yildirim, C. (2009).
\newblock Primes in tuples i.
\newblock {\em Ann. Math.}, 170(2):819--862.

\bibitem[Erd\"{o}s, 1962]{erdoes_1962}
Erd\"{o}s, P. (1962).
\newblock On trigonometric sums with gaps.
\newblock {\em Publ. Math. Inst. Hung. Acad. Sci., Ser. A}, 7:37--42.

\bibitem[Erd\"{o}s and G\'{a}l, 1955]{erdoes_gal_1955}
Erd\"{o}s, P. and G\'{a}l, I. (1955).
\newblock On the law of iterated logarithm i + ii.
\newblock {\em Nederl. Akad. Wetensch. Proc. Ser. A.}, 17(58):65--84.

\bibitem[Fr\"{o}berg, 1968]{froeberg1968}
Fr\"{o}berg, C.-E. (1968).
\newblock On the prime zeta function.
\newblock {\em BIT}, 8:187--202.

\bibitem[Gaposhkin, 1966]{gaposhkin_1966}
Gaposhkin, V.~F. (1966).
\newblock Lacunary series and independent functions.
\newblock {\em Uspehi Mat. Nauk. 21}, 132(6):3--82.

\bibitem[Gerver, 1970a]{gerver1970}
Gerver, J. (1970a).
\newblock The differentiability of the riemann function at certain rational
  multiples of $\pi$.
\newblock {\em Amer. J. Math.}, 92:33--–55.

\bibitem[Gerver, 1970b]{gerver1970b}
Gerver, J. (1970b).
\newblock More on the differentiability of the riemann function.
\newblock {\em Amer. J. Math.}, 93:33--–41.

\bibitem[Hardy and Littlewood, 1912]{hardy_littlewood1912}
Hardy and Littlewood (1912).
\newblock Contributions to the arithmetic theory of series.
\newblock {\em Proceedings of the London Mathematical Society}, 11(2):411--478.

\bibitem[Hardy, 1916]{hardy1916}
Hardy, G.~H. (1916).
\newblock Weierstrass's non-differentiable function.
\newblock {\em Transactions of the American Mathematical Society},
  17(3):301--325.

\bibitem[Jaffard, 2010]{jaffard2010}
Jaffard, S. (2010).
\newblock Pointwise and directional regularity of nonharmonic fourier series.
\newblock {\em Applied and Computational Harmonic Analysis}, 22(3):251--266.

\bibitem[Kahane, 1997]{kahane1997}
Kahane, J.-P. (1997).
\newblock A century of interplay between taylor series, fourier series and
  brownian motion.
\newblock {\em Bull. London Math. Soc.}, 29:257--279.

\bibitem[Landau and Walfisz, 1920]{landau_walfisz1920}
Landau, E. and Walfisz, A. (1920).
\newblock \"{U}ber die nichfortsetzbarkeit einiger durch dirichletsche reihen
  definierter funktionen.
\newblock {\em Rend. Circ. Math. Palermo}, 44:82--86.

\bibitem[Philipp and Stout, 1975]{philipp_stout_1975}
Philipp, W. and Stout, W.~F. (1975).
\newblock {\em Almost Sure Invariance Principles for Partial Sums of Weakly
  Dependent Random Variables}.
\newblock Mem. Am. Math. Soc. AMS.

\bibitem[Rosser and Schoenfeld, 1962]{rosser1962}
Rosser, J. and Schoenfeld, L. (1962).
\newblock Approximate formulas for some functions of prime numbers.
\newblock {\em Illinois J. Math.}, 6(1):64--94.

\bibitem[Salem and Zygmund, 1947]{salem_zygmund1947}
Salem, R. and Zygmund, A. (1947).
\newblock On lacunary trigonometric series.
\newblock {\em Proc. Nat. Acad. Sci. U.S.A.}, 33:333--338.

\bibitem[Salem and Zygmund, 1948]{salem_zygmund1948}
Salem, R. and Zygmund, A. (1948).
\newblock On lacunary trigonometric series.
\newblock {\em Proc. Nat. Acad. Sci. U.S.A.}, 34:54--62.

\bibitem[Schatte, 1988]{schatte_1988}
Schatte, P. (1988).
\newblock On a law of iterated logarithm for sums mod 1 with applications to
  benford's law.
\newblock {\em Prob. theory Rel. Fields}, 77:167--178.

\bibitem[Vartziotis and Wipper, 2016]{Vartziotis_Wipper_2016}
Vartziotis, D. and Wipper, J. (2016).
\newblock {The fractal nature of an approximate prime counting function}.
\newblock {\em ArXiv e-prints}.

\bibitem[Weiss, 1959]{weiss_1959}
Weiss, M. (1959).
\newblock The law of the iterated logarithm for lacunary trigonometric series.
\newblock {\em Trans. Amer. Math. Soc.}, 91:444--469.

\end{thebibliography}
\end{document}